\documentclass{amsart}

\newtheorem{theorem}{Theorem}
\newtheorem{lemma}[theorem]{Lemma}
\newtheorem{corollary}[theorem]{Corollary}

\begin{document}

\title{Combinatorial Yamabe flow on hyperbolic surfaces with boundary}

\author{Ren Guo}

\address{School of Mathematics, University of Minnesota, Minneapolis, MN, 55455}

\email{guoxx170@math.umn.edu}



\begin{abstract} This paper studies the combinatorial Yamabe flow on hyperbolic surfaces with boundary. It is proved by applying a variational principle that the length of boundary components is uniquely determined by the combinatorial conformal factor. The combinatorial Yamabe flow is a gradient flow of a concave function. The long time behavior of the flow and the geometric meaning is investigated. 
\end{abstract}

\maketitle

\section{Introduction}

\subsection{Piecewise flat metrics}
In trying to develop the analogous piecewise linear conformal geometry, Luo studied the combinatorial Yamabe problem for piecewise flat metrics on triangulated surfaces \cite{Luo04}. We summarize a part of this work in the following. Suppose $\Sigma$ is a connected orientable closed surface with a triangulation $T$ so that $V,E,F$ are sets of all vertexes, edges and triangles in $T.$ We identify a vertex of $T$ with an index, an edge of $T$ with a pair of indexes and a triangle of $T$ with a triple of indexes. This means $V=\{1,2,...n\}, E=\{ij\ |\ i,j\in V\}$ and $F=\{ijk\ |\ i,j,k\in V\},$ where $n$ is the number of vertexes.

A  piecewise flat metric on $(\Sigma, T)$ is identified with a vector indexed by the set of edges $E$. More precisely, it is an assignment to each edge a positive number such that the three numbers assigned to the three edges of a triangle satisfy the triangle inequality. Equipped with a piecewise flat metric, each triangle of $T$ can be realized as a Euclidean triangle and $\Sigma$ can be realized as a Euclidean polyhedral surface. 

Let's fix a piecewise flat metric on $(\Sigma, T)$ as $l^0\in \mathbb{R}_{>0}^{|E|}$. The assignment to the edge $ij$ is denoted by $l^0_{ij}.$ A \it combinatorial conformal factor \rm on $(\Sigma, T)$ is a vector $w=(w_1,w_2,...,w_n)\in \mathbb{R}^n$ which assigns each vertex $i\in V$ a number $w_i.$ (In \cite{Luo04}, the notation $u_i$ is used, where $u_i=e^{w_i}$.) From a combinatorial conformal factor, we obtained a new vector $l\in \mathbb{R}_{>0}^{|E|}$ as follows:
\begin{align}\label{fml:e-conformal}
l_{ij}=e^{w_i+w_j}l^0_{ij}
\end{align}
for each edge $ij\in E.$

Let $\mathcal{W}_E$ be the space of combinatorial conformal factors $w$ such that each vector $l$ corresponding to a vector $w\in \mathcal{W}_E$ is indeed a piecewise flat metric. In other words, the triangle inequality holds for each triangle under the assignment $l.$ Obviously, $\mathcal{W}_E$ depends on the initial metric $l^0.$ 

For a vector $w\in \mathcal{W}_E,$ the vector $l$ corresponding to $w$ is a piecewise flat metric. Each triangle of $T$ is realized as a Euclidean triangle under the metric $l$. At a vertex $i,$ the curvature of the metric $l$ is defined as follows. Let $\alpha^i_{jk}$ be the inner angle of the triangle $ijk\in F$ between the edges $ij$ and $ik.$ Then
$$K_i=2\pi-\sum_{ijk\in F} \alpha^i_{jk}$$
is the curvature at the vertex $i$.

This produces a map 
$$
\begin{array}{ccccccc}
\psi_E: \mathcal{W}_E & \to & \mathbb{R}^n\\
(w_1,w_2,...,w_n) & \mapsto &  (K_1,K_2,...,K_n)
\end{array}
$$
sending a combinatorial conformal factor to the curvature.

\begin{theorem}[Luo]
The map $\psi_E$ is a local homeomorphism.
\end{theorem}

The theorem is proved by applying a variational principle. An local convex energy function is constructed using the derivative cosine law. And $\psi_E$ truns out to be the gradient of the energy function.

Motivated by establishing a discrete Uniformization Theorem, Luo introduced the combinatorial Yamabe flow

\begin{align}\label{fml:eflow}
\left\{
\begin{array}{ccccccc}
\frac{dw_i(t)}{dt}=-K_i(t),\\
w_i(0)=0.
\end{array}
\right.
\end{align}

\begin{corollary}[Luo]
The combinatorial Yamabe flow (\ref{fml:eflow}) is the negative gradient flow of a locally convex function in terms of $w$. And $\sum_{i=1}^n K_i^2(t)$ is decreasing in time $t$. 
\end{corollary}

\subsection{Related work} 

Motivated by the application in computer graphics, Springborn, Schr\"{o}der, and Pinkall \cite{SSP08} considered this combinatorial conformal change of piecewise flat metrics (\ref{fml:e-conformal}). They found an explicit formula of the energy function. Glickenstein \cite{Gli05a,Gli05b} studied the combinatorial Yamabe flow on 3-dimensional piecewise flat manifolds relating to the ball packing metric of Cooper and Rivin \cite{CR96}. Recently Glickenstein \cite{Gli09} set the theory of combinatorial Yamabe flow of piecewise flat metric in a broader context including the theory of circle packing on surfaces. This combinatorial conformal change of metrics has appeared in physic literature \cite{RW84} and numerical analysis literature \cite{Ker96,PC98}. We were informed by Luo in 2009 that Springborn considered the combinatorial conformal change of hyperbolic metric on a triangulated closed surface. He introduced the combinatorial conformal change as 
\begin{align}\label{fml:h-conformal}
\sinh\frac{l_{ij}}{2}=e^{w_i+w_j}\sinh\frac{l^0_{ij}}{2}.
\end{align}

\subsection{Hyperbolic metrics}

In this paper we study the combinatorial Yamabe flow on hyperbolic surfaces with geodesic boundary. Let $\Sigma$ be a connected orientable compact surface with $n$ boundary components. The set of boundary components is $B=\{1,2,...n\}$ where a boundary component is identified with an index.

A colored hexagon is a hexagon with three non-pairwise adjacent edges labeled by red and the opposite edges labeled by black. Take a finite disjoint union of colored hexagons and identify all red edges in pairs by homeomorphisms. The quotient is a compact surface with non-empty boundary together with an \it ideal triangulation\rm. The faces in the ideal triangulation are quotients of the hexagons. The quotients of red edges are called the edges of the ideal triangulation while the quotients of black edges are called the boundary arcs. 

It is well-known that each connected orientable compact surface $\Sigma$ of non-empty boundary and negative Euler characteristic admits an ideal triangulation. 

Let $T$ be an ideal triangulation of $\Sigma$. Since an edge of $T$ connects two boundary components of $\Sigma$, an edge of $T$ is indentified with a pair of indexes. The set of edges of $T$ is $E=\{ij\ |\ i,j\in B\}.$ Since a face in $T$ is determined by its boundary arcs in three boundary components, a face of $T$ is indentified with a triple of indexes. The set of faces of $T$ is $F=\{ijk\ |\ i,j,k\in B\}.$ 

A hyperbolic metric on $(\Sigma,T)$ is identified with a vector indexed by the set of edges $E$. More precisely, it is an assignment to each edge a positive number. It is well-known that \cite{Bus92}, for any three positive numbers, there exists a hyperbolic right-angled hexagon the length of whose three non-pairwise adjacent edges are the three numbers. Furthermore, the hexagon is unique up to isometry. Therefore, for vector in $\mathbb{R}_{>0}^{|E|}$, each face of $F$ can be realized as a unique hyperbolic right-angled hexagon and the surface $\Sigma$ can be realized as a hyperbolic surface with geodesic boundary.

Let's fix a hyperbolic metric on $(\Sigma,T)$ as $l^0\in \mathbb{R}_{>0}^{|E|}$. The assignment to the edge $ij$ is denoted by $l^0_{ij}.$ A \it combinatorial conformal factor \rm on $(\Sigma, T)$ is a vector $w=(w_1,w_2,...,w_n)\in \mathbb{R}^n$ which assigns each boundary component $i\in B$ a number $w_i.$ From a combinatorial conformal factor, we obtained a new assignment $l\in \mathbb{R}^{|E|}$ as follows
\begin{align}\label{fml:hb-conformal}
\cosh\frac{l_{ij}}{2}=e^{w_i+w_j}\cosh\frac{l^0_{ij}}{2}
\end{align}
for any edge $ij\in E$. This definition (\ref{fml:hb-conformal}) is an analogue of Springborn's definition of combinatorial conformal change of hyperbolic metrics on triangulated closed surfaces (\ref{fml:h-conformal}). 

Denote by $\mathcal{W}$ the set of combinatorial conformal factors such that the corresponding assignment is positive, i.e.,   $l\in \mathbb{R}_{>0}^{|E|}$. Therefore, for $w\in\mathcal{W}$, we obtained a new hyperbolic metric on $(\Sigma,T)$. Each face of $T$ is a hyperbolic right-angled hexagon. Let $\theta^i_{jk}$ be the length of the boundary arc in the boundary component $i$ of the face $ijk\in F$. Then the length of the boundary component $i$ is
$$B_i=\sum_{ijk\in F} \theta^i_{jk}.$$

This produces a map 
$$
\begin{array}{ccccccc}
\psi: \mathcal{W} & \to & \mathbb{R}^n\\
(w_1,w_2,...,w_n) & \mapsto &  (B_1,B_2,...,B_n)
\end{array}
$$
sending a combinatorial conformal factor to the length of boundary components.

\begin{theorem}\label{thm:homo}
The map $\psi$ is a homeomorphism.
\end{theorem}

This is a result of global rigidity while Theorem 1 is a result of local rigidity.

We also consider the combinatorial Yamabe flow in this situation

\begin{align}\label{fml:hflow}
\left\{
\begin{array}{ccccccc}
\frac{dw_i(t)}{dt}=B_i(t),\\
w_i(0)=0.
\end{array}
\right.
\end{align}

\begin{corollary}\label{thm:flow}
The combinatorial Yamabe flow (\ref{fml:hflow}) is the gradient flow of a concave function in terms of $w$. And $\sum_{i=1}^n B_i^2(t)$ is decreasing in time $t$. 
\end{corollary}
 
We investigate the long time behavior of the flow. 

\begin{theorem}\label{thm:cusp}
 The combinatorial Yamabe flow (\ref{fml:hflow}) has a solution for $t\in[0,\infty)$. Along the flow (\ref{fml:hflow}), any initial hyperbolic surface with geodesic boundary converges to a complete hyperbolic surface with cusps.
\end{theorem} 

\subsection{Variational principle}

The approach of variational principle of studying polyhedral surfaces was introduced by Colin de Verdi\'ere
\cite{CdV91} in his proof of Andreev-Thurston's circle packing theorem. Since then, many works about variational principles on polyhedral surfaces have appeared. For example, see \cite{Bra92,Riv94,Lei02,CL03,BS04,Luo06,Guo07,Spr08,GL09,Guo09} and others.

\subsection{Organization of the paper}

Theorem \ref{thm:homo} is proved in section 2. Corollary \ref{thm:flow} and Theorem \ref{thm:cusp} are proved in section 3.

\section{Homeomorphism}

\subsection{Space of combinatorial conformal factors}

Let $l^0\in \mathbb{R}_{>0}^{|E|}$ be a fixed hyperbolic metric on $(\Sigma,T)$. We investigate the space of combinatorial conformal factor such that the corresponding new assignment is a hyperbolic metric. For a face $ijk\in F,$ denote by $\mathcal{W}^{ijk}$ the space of vectors $(w_i,w_j,w_k)$ such that $l_{jk},l_{ki}$ and $l_{ij}$ are positive.

\begin{lemma} \label{thm:space}
 $\mathcal{W}^{ijk}$ is a convex polytope.
\end{lemma}

\begin{proof}
By definition (\ref{fml:hb-conformal}) $\cosh\frac{l_{ij}}{2}=e^{w_i+w_j}\cosh\frac{l^0_{ij}}{2}$. The only requirement is $l_{ij}>0.$ Hence 
$$w_i+w_j>-\ln \cosh\frac{l^0_{ij}}{2}.$$
Similar inequalities hold for $w_j+w_k$ and $w_k+w_i$.
Therefore $\mathcal{W}^{ijk}$ is the intersection of three half space. 
\end{proof}

\begin{corollary} \label{thm:wspace}
The space $\mathcal{W}$ is a convex polytope.
\end{corollary}

\begin{proof}
$\mathcal{W}=\cap_{ijk\in F} \mathcal{W}^{ijk}.$
\end{proof}

\subsection{Energy function}

Let's focus on one face $ijk\in F.$ When $(w_i,w_j,w_k)\in \mathcal{W}^{ijk},$ from (\ref{fml:hb-conformal}), we obtain a hyperbolic right-angled hexagon whose non-pairwise adjacent edges have length $l_{jk}, l_{ki}$ and $l_{ij}$. Let $\theta^i_{jk}, \theta^j_{ki}$ and $\theta^k_{ij}$ be the length of the hyperbolic arcs opposite to the edges $jk, ki$ and $ij$ of this hexagon. Therefore $\theta^i_{jk}, \theta^j_{ki}$ and $\theta^k_{ij}$ are functions of $w_i,w_j,w_k$.

\begin{lemma}\label{thm:symmetry} 
The Jacobian matrix of functions $\theta^i_{jk}, \theta^j_{ki}, \theta^k_{ij}$ in terms of $w_i,w_j,w_k$ is symmetric.
\end{lemma}

\begin{proof}
The cosine law for hyperbolic right-angled hexagon induces the derivative cosine law:
\begin{multline*}
\left(
\begin{array}{ccc}
d\theta^i_{jk}\\
d\theta^j_{ki}\\
d\theta^k_{ij}
\end{array}\right)
=\frac{-1}{\sinh \theta^k_{ij} \sinh l_{ki}\sinh l_{jk}} \left(
\begin{array}{ccc}
\sinh l_{jk}&0&0 \\
0&\sinh l_{ki}&0 \\
0&0&\sinh l_{ij}
\end{array}
\right)\\
 \left(
\begin{array}{ccc}
-1                &   \cosh \theta^k_{ij} &\cosh \theta^j_{ki} \\
\cosh \theta^k_{ij}&  -1                   &\cosh \theta^i_{jk} \\
\cosh \theta^j_{ki}&  \cosh \theta^i_{jk} &-1
\end{array}
\right) \left(
\begin{array}{ccc}
dl_{jk} \\
dl_{ki} \\
dl_{ij}
\end{array}\right).
\end{multline*}

By differentiating the two sides of the equation (\ref{fml:hb-conformal}), $\cosh\frac{l_{ij}}{2}=e^{w_i+w_j}\cosh\frac{l^0_{ij}}{2}$, we obtain
$$dl_{ij}=\frac{2\sinh l_{ij}}{\cosh l_{ij}-1}(d w_i+d w_j).$$ 
Similar formulas hold for $d l_{jk}$ and $d l_{ki}$. Then we have 

\begin{multline*}
\left(
\begin{array}{ccc}
d\theta^i_{jk}\\
d\theta^j_{ki}\\
d\theta^k_{ij}
\end{array}\right)
=\frac{-2}{\sinh \theta^k_{ij} \sinh l_{ki}\sinh l_{jk}} \left(
\begin{array}{ccc}
\sinh l_{jk} & 0          &0 \\
0            &\sinh l_{ki}&0 \\
0            &0           &\sinh l_{ij}
\end{array}
\right)\\
 \left(
\begin{array}{ccc}
-1                &   \cosh \theta^k_{ij} &\cosh \theta^j_{ki} \\
\cosh \theta^k_{ij}&  -1                   &\cosh \theta^i_{jk} \\
\cosh \theta^j_{ki}&  \cosh \theta^i_{jk} &-1
\end{array}
\right) 
\left(
\begin{array}{ccc}
\frac{\sinh l_{jk}}{\cosh l_{jk}-1} & 0                                  &0 \\
0                                    &\frac{\sinh l_{ki}}{\cosh l_{ki}-1}&0 \\
0                                    &0                                   &\frac{\sinh l_{ij}}{\cosh l_{ij}-1}
\end{array}
\right)\\
\left(
\begin{array}{ccc}
0 & 1 &1 \\
1 & 0 &1 \\
1 & 1 &0
\end{array}
\right)
\left(
\begin{array}{ccc}
d w_i \\
d w_j \\
d w_k
\end{array}\right).
\end{multline*}
 
For simplicity of the notations, the above formula is written as 
$$
\left(
\begin{array}{ccc}
d\theta^i_{jk}\\
d\theta^j_{ki}\\
d\theta^k_{ij}
\end{array}\right)
=\frac{-2}{\sinh \theta^k_{ij} \sinh l_{ki}\sinh l_{jk}}
M
\left(
\begin{array}{ccc}
d w_i \\
d w_j \\
d w_k
\end{array}\right)
$$
where $M$ is a product of four matrixes. To prove the lemma, it is enough to show that the matrix $M$ is symmetric. 

Represent $\cosh\theta^i_{jk}, \cosh\theta^j_{ki}, \cosh\theta^k_{ij}$ as functions of $\cosh l_{jk}, \cosh l_{ki}, \cosh l_{ij}$ using the cosine law. For simplicity of the notations, let $a:=\cosh l_{jk},b:=\cosh l_{ki},c:=\cosh l_{ij}$, Then we have
$$
M=
\left(
\begin{matrix}
\displaystyle \frac{c+ab}{b-1} + \frac{b+ac}{c-1}  
& \displaystyle \frac{a+b-c+1}{c-1}                
& \displaystyle \frac{a+c-b+1}{b-1}  \\
\displaystyle\frac{a+b-c+1}{c-1}                  
&\displaystyle \frac{c+ab}{a-1} + \frac{a+bc}{c-1} 
&\displaystyle \frac{b+c-a+1}{a-1} \\
\displaystyle \frac{a+c-b+1}{b-1}                  
&\displaystyle \frac{b+c-a+1}{a-1}                 
&\displaystyle \frac{b+ac}{a-1} + \frac{a+bc}{b-1}
\end{matrix}
\right).
$$
\end{proof}

\begin{lemma}\label{thm:definite} 
The Jacobian matrix of functions $\theta^i_{jk}, \theta^j_{ki}, \theta^k_{ij}$ in terms of $w_i,w_j,w_k$ is negative definite.
\end{lemma}

\begin{proof} 
We only need to show that the matrix $M$ is positive definite. Let $M_{rs}$ be the entry of $M$ at $r-$th row and $s-$th column. First, $M_{11}>0.$ Second, 
\begin{align*}
M_{11}-M_{12}= \frac{c+ab}{b-1} + a+1 >0,\\
M_{22}-M_{21}= \frac{c+ab}{a-1} + b+1 >0.
\end{align*}
Then $M_{11}M_{22}-M_{12}M_{21}>0.$ Third,
\begin{multline*}
\det M=\sinh l_{jk} \sinh l_{ki}\sinh l_{ij}\cdot(\sinh \theta^i_{jk} \sinh \theta^j_{ki} \sinh l_{ij})^2 \\
\cdot \frac{\sinh l_{jk}}{\cosh l_{jk}-1} \frac{\sinh l_{ki}}{\cosh l_{ki}-1} \frac{\sinh l_{ij}}{\cosh l_{ij}-1}
\cdot 2>0. 
\end{multline*}
\end{proof}

\begin{corollary}\label{thm:function}
The differential 1-form $\theta^i_{jk} d w_i+ \theta^j_{ki} d w_j + \theta^k_{ij} d w_k$ is closed on $\mathcal{W}^{ijk}$. For any $c\in \mathcal{W}^{ijk},$ the integral 
$$\mathcal{E}(w_i,w_j,w_k)=\int_c^{(w_i,w_j,w_k)}(\theta^i_{jk} d w_i+ \theta^j_{ki} d w_j + \theta^k_{ij} d w_k)$$
is a strictly concave function on $\mathcal{W}^{ijk}$ satisfying 

\begin{align*}
\frac{\partial \mathcal{E}}{\partial w_i} =\theta^i_{jk}, 
\frac{\partial \mathcal{E}}{\partial w_j} =\theta^j_{ki}, 
\frac{\partial \mathcal{E}}{\partial w_k} =\theta^k_{ij}.
\end{align*}
\end{corollary} 

\begin{proof} The differential 1-form is closed due to Lemma \ref{thm:symmetry}. Since $\mathcal{W}^{ijk}$ is connected and simply connected due to Lemma \ref{thm:space}, then the function $\mathcal{E}(w_i,w_j,w_k)$ is well defined, i.e., independent the path of integration. By Lemma \ref{thm:definite}, the Hessian matrix of $\mathcal{E}(w_i,w_j,w_k)$ is negative definite. 
\end{proof}

\subsection{Homeomorphism}
In this subsection we prove Theorem \ref{thm:homo}.The following two lemmas are needed.

The first one is well-known in analysis.

\begin{lemma}\label{thm:convex} 
Suppose $X$ is an open convex set in $\mathbb{R}^N$ and $f: X\to \mathbb{R}$ a smooth function.
If the Hessian matrix of $f$ is positive definite for all $x \in X$, then the
gradient $\nabla f:X\to\mathbb{R}^N$ is a smooth embedding.
\end{lemma}

\begin{lemma} \label{thm:converge}
For a family of combinatorial conformal factor $w^{(m)}\in \mathcal{W}$, if 
\newline
$\lim_{m\to \infty}w_k^{(m)}=\infty$ for some $k$, then $\lim_{m\to\infty}B_k^{(m)}=0$ and the convergence is independent of the values of $\lim_{m\to\infty} w_r^{(m)}$ for $r\neq k.$
\end{lemma}

\begin{proof}
By definition (\ref{fml:hb-conformal}),
$$\cosh l_{jk}^{(m)}=e^{2w_j^{(m)}+2w_k^{(m)}}c_1-1,$$
$$\cosh l_{ki}^{(m)}=e^{2w_k^{(m)}+2w_i^{(m)}}c_2-1,$$
$$\cosh l_{ij}^{(m)}=e^{2w_i^{(m)}+2w_j^{(m)}}c_3-1,$$
where $c_1=2\cosh^2 \frac{l^0_{jk}}2, c_2=2\cosh^2 \frac{l^0_{ki}}2, c_3=2\cosh^2 \frac{l^0_{ij}}2.$
Then
\begin{align*}
&\ \lim_{m\to\infty} \cosh (\theta^k_{ij})^{(m)}\\ 
= &\ \lim_{m\to\infty} 
\frac{\cosh l_{ij}^{(m)} + \cosh l_{jk}^{(m)}\cosh l_{ki}^{(m)}}{\sinh l_{jk}^{(m)}\sinh l_{ki}^{(m)}}\\
= &\ \lim_{m\to\infty} 
\frac{e^{2w_i^{(m)}+2w_j^{(m)}}c_3-1 + (e^{2w_j^{(m)}+2w_k^{(m)}}c_1-1)(e^{2w_k^{(m)}+2w_i^{(m)}}c_2-1)}
{e^{w_j^{(m)}+w_k^{(m)}}\sqrt{e^{2w_j^{(m)}+2w_k^{(m)}}c^2_1-2c_1}\cdot
e^{w_k^{(m)}+w_i^{(m)}}\sqrt{e^{2w_k^{(m)}+2w_i^{(m)}}c^2_2-2c_2}}\\
= &\ \lim_{m\to\infty}
\frac{c_3-e^{-2w_i^{(m)}-2w_j^{(m)}}}
{e^{4w_k^{(m)}}\sqrt{(c^2_1-2c_1e^{-2w_j^{(m)}-2w_k^{(m)}})(c^2_2-2c_2e^{-2w_i^{(m)}-2w_k^{(m)}})}}\\
&\ \ \ \ \ \ \ \ \ \ \ \ \ \ \ \ \ \ + \lim_{m\to\infty} 
\frac{(c_1-e^{-2w_j^{(m)}-2w_k^{(m)}})(c_2-e^{-2w_i^{(m)}-2w_k^{(m)}})}
{\sqrt{(c^2_1-2c_1e^{-2w_j^{(m)}-2w_k^{(m)}})(c^2_2-2c_2e^{-2w_i^{(m)}-2w_k^{(m)}})}}\\
= &\ 0+1.
\end{align*}
Hence $\lim_{m\to\infty}(\theta^k_{ij})^{(m)}=0$ independent the values of $\lim_{m\to\infty} w_r^{(m)}$ for $r\neq k.$
Thus $\lim_{m\to\infty}B_k^{(m)}=\lim_{m\to\infty}\sum_{ijk\in F}(\theta^k_{ij})^{(m)}=0.$
\end{proof}

\begin{proof}[Proof of Theorem \ref{thm:homo}]

Let $l^0\in \mathbb{R}^{|E|}_{>0}$ be a fixed hyperbolic metric on $(\Sigma,T).$ For any combinatorial conformal factor $w\in \mathcal{W}$, we obtain a new hyperbolic metric $l\in \mathbb{R}^{|E|}_{>0}$. By Corollary \ref{thm:function}, for each face $ijk\in F$, there is a function $\mathcal{E}(w_i,w_j,w_k).$ Define a function $\overline{\mathcal{E}}:\mathcal{W}\to \mathbb{R}$ by
 $$\overline{\mathcal{E}}(w_1,w_2,...,w_n)=\sum_{ijk \in F} \mathcal{E}(w_i, w_j, w_k)$$ where the sum is
over all faces in $F$. By Corollary \ref{thm:function}, $\overline{\mathcal{E}}$ is strictly concave on $\mathcal{W}$ and 
\begin{align}\label{fml:gradient}
\frac{\partial \overline{\mathcal{E}}}{\partial w_i}
=\sum_{ijk \in F}\frac{\partial \mathcal{E}(w_i, w_j, w_k)}{\partial w_i} 
=\sum_{ijk \in F} \theta^i_{jk}
=B_i.
\end{align}
That means the gradient of $\overline{\mathcal{E}}$ is exactly the map $\psi$ sending a combinatorial conformal factor $w$ to the corresponding length of boundary components. Thus $\psi$ is a smooth embedding due to Lemma \ref{thm:convex}.

To show that $\psi$ is a homeomorphism, we will prove that $\psi(\mathcal{W})$
is both open and closed in $\mathbb{R}^n_{>0}.$ 

Since $\psi$ is a smooth embedding, $\psi(\mathcal{W})$ is open in $\mathbb{R}^n_{>0}$.

To show that $\psi(\mathcal{W})$ is closed in $\mathbb{R}^n_{>0},$ take a sequence
of combinatorial conformal factor $w^{(m)}$ in $\mathcal{W}$ such that
$\lim_{m\to\infty} (B_1^{(m)},B_2^{(m)},...,B_n^{(m)})\in\mathbb{R}^n_{>0}.$ To
prove the closeness, it is sufficient to show that there is a
subsequence of $w^{(m)}$ whose limit is in $\mathcal{W}$. 

Suppose otherwise, there is a subsequence, still denoted by $w^{(m)}$, so that
its limit is on the boundary of
$\mathcal{W}$. The first possibility is that there is some $k$ such that 
$\lim_{m\to\infty} w_k^{(m)}=\infty.$ By Lemma \ref{thm:converge}, $\lim_{m\to\infty}B_k^{(m)}=0$. 
This contradicts the assumption that $\lim_{m\to\infty} (B_1^{(m)},B_2^{(m)},...,B_n^{(m)})\in\mathbb{R}^n_{>0}.$

The second possibility is that $\lim_{m\to \infty}(w_i^{(m)}+w_j^{(m)})=-\ln \cosh\frac{l^0_{ij}}{2}$ for some edge $ij.$ That means $\lim_{m\to \infty}l_{ij}^{(m)}=0.$ For the face $ijk\in F$, we have 
\begin{align*}
\lim_{m\to\infty}(\theta^i_{jk})^{(m)} =&\  
\lim_{m\to\infty} 
\frac{\cosh l_{jk}^{(m)} + \cosh l_{ik}^{(m)}\cosh l_{ij}^{(m)}}{\sinh l_{ik}^{(m)}\sinh l_{ij}^{(m)}}\\
\geq &\ \lim_{m\to\infty} 
\frac{\cosh l_{ik}^{(m)}\cosh l_{ij}^{(m)}}{\sinh l_{ik}^{(m)}\sinh l_{ij}^{(m)}}\\
\geq &\ \lim_{m\to\infty} 
\frac{\cosh l_{ij}^{(m)}}{\sinh l_{ij}^{(m)}}=\infty.
\end{align*}
Therefore $\lim_{m\to\infty}(\theta^i_{jk})^{(m)}=\infty$ and $\lim_{m\to\infty}B_i^{(m)}=\infty.$ This contradicts the assumption that $\lim_{m\to\infty} (B_1^{(m)},B_2^{(m)},...,B_n^{(m)})\in\mathbb{R}^n_{>0}.$
\end{proof}

\section{Flow}

\begin{proof}[Proof of Corollary \ref{thm:flow}]
The combinatorial Yamabe flow (\ref{fml:hflow}) is the gradient flow of the concave function $\overline{\mathcal{E}}(w_1,w_2,...,w_n)$ due to the equation (\ref{fml:gradient}). 

Since 
\begin{align*}
\frac{d B_i(t)}{dt}=&\ \sum_{ijk\in F }\frac{d \theta^i_{jk}}{dt}\\
=&\ \sum_{ijk\in F } (\frac{d \theta^i_{jk}}{dw_i}\frac{dw_i}{dt}
+\frac{d \theta^i_{jk}}{dw_j}\frac{dw_j}{dt}
+\frac{d \theta^i_{jk}}{dw_k}\frac{dw_k}{dt})\\
=&\ \sum_{ijk\in F } (\frac{d \theta^i_{jk}}{dw_i}B_i
+\frac{d \theta^i_{jk}}{dw_j}B_j
+\frac{d \theta^i_{jk}}{dw_k}B_k),
\end{align*}
we have
\begin{align*}
&\ \frac12\frac{d}{dt}\sum_{i=1}^n B_i^2(t) \\
= &\ \sum_{i=1}^n B_i \frac{d B_i}{dt}\\
=&\  \sum_{ijk\in F }(\frac{d \theta^i_{jk}}{dw_i}B_i^2
+ \frac{d \theta^j_{ki}}{dw_j}B_j^2
+ \frac{d \theta^k_{ij}}{dw_k}B_k^2\\
&\ \ \ \ \ \ \ \ \ +(\frac{d \theta^j_{ki}}{dw_j}+\frac{d \theta^k_{ij}}{dw_k})B_jB_k
+(\frac{d \theta^k_{ij}}{dw_k}+\frac{d \theta^i_{jk}}{dw_i})B_kB_i
+(\frac{d \theta^i_{jk}}{dw_i}+\frac{d \theta^j_{ki}}{dw_j})B_iB_j)\\
=&\ \sum_{ijk\in F } (B_i,B_j,B_k) 
\frac{\partial(\theta^i_{jk},\theta^j_{ki},\theta^k_{ij})}{\partial(w_i,w_j,w_k)} (B_i,B_j,B_k)^T.
\end{align*}
By Lemma \ref{thm:definite}, the Jacobian matrix $\frac{\partial(\theta^i_{jk},\theta^j_{ki},\theta^k_{ij})}{\partial(w_i,w_j,w_k)}$ is negative definite. 

Hence $\frac12\frac{d}{dt}\sum_{i=1}^n B_i^2(t)<0.$ Therefore $\sum_{i=1}^n B_i^2(t)$ is decreasing in $t$.
\end{proof}

\begin{proof}[Proof of Theorem \ref{thm:cusp}]
Since $B_i(t)>0,$ then $w_i(t)>w_i(0)=0.$ 

For any $L<\infty,$ we claim that $\lim_{t\to L}w_i(t)<\infty.$ 

Otherwise, if $\lim_{t\to L}w_i(t)=\infty$ for some $L<\infty,$ then by Lemma \ref{thm:converge}, we see $\lim_{t\to L}B_i(t)=0.$ Therefore, for any $\epsilon>0,$ there exists some $\delta>0$ such that when $t\in(L-\epsilon,L),$ the inequalities $0<B_i(t)<\epsilon$ holds. Hence, by the flow (\ref{fml:hflow}), $0<\frac{dw_i(t)}{dt}<\epsilon$ holds. Thus $w_i(0)<w_i(t)<\epsilon t< \epsilon L.$ This contradicts to the assumption that $\lim_{t\to L}w_i(t)=\infty$.

Hence the solution of the flow (\ref{fml:hflow}) exists for all time $t\in[0,\infty).$

To obtain the geometric picture, we claim that $\lim_{t\to\infty}B_i(t)=0$ for each $1 \leq i\leq n$. There are two cases to consider. 

First, if $\lim_{t\to\infty}w_i(t)=\infty,$ by Lemma \ref{thm:converge}, $\lim_{t\to\infty}B_i(t)=0$ holds.

Second, if $\lim_{t\to\infty}w_i(t)<\infty,$ we claim $\lim_{t\to\infty}B_i(t)=0$ still holds.
Otherwise, $\lim_{t\to\infty}B_i(t)=a>0.$ Then, for any $\epsilon\in (0,a),$ there exists some $P>0$ such that when $t>P,$ the inequality $B_i(t)>a-\epsilon$ holds. Hence, by the flow (\ref{fml:hflow}), $\frac{dw_i}{dt}>a-\epsilon$ holds. Therefore $w_i(t)>(a-\epsilon)t.$ This contradicts to the assumption $\lim_{t\to\infty}w_i(t)<\infty.$

Once $\lim_{t\to\infty}B_i(t)=0$ for each $1 \leq i\leq n$, we have $\lim_{t\to\infty}\theta^i_{jk}(t)=0$ for any face $ijk\in F.$ Hence, for any edge $ij,$
\begin{align*}
\lim_{t\to\infty}l_{ij}(t) =&\  
\lim_{t\to\infty} 
\frac{\cosh \theta^k_{ij}(t) + \cosh \theta^i_{jk}(t)\cosh \theta^j_{ki}(t)}{\sinh \theta^i_{jk}(t)\sinh \theta^j_{ki}(t)}\\
\geq &\ \lim_{t\to\infty} 
\frac{\cosh \theta^i_{jk}(t)\cosh \theta^j_{ki}(t)}{\sinh \theta^i_{jk}(t)\sinh \theta^j_{ki}(t)}\\
\geq &\ \lim_{t\to\infty} 
\frac{\cosh \theta^i_{jk}(t)}{\sinh \theta^i_{jk}(t)}=\infty.
\end{align*}
Therefor each hyperbolic right-angled hexagon converges to a hyperbolic ideal triangle. 
\end{proof}

\section*{Acknowledgment} 

The author would like to thank Feng Luo for encouragement and helpful conversations.

\end{document}